\documentclass{article} 

\usepackage{graphicx}
\usepackage{subfigure}
\usepackage{amsmath}
\usepackage{amsthm}
\usepackage{amssymb}
\usepackage{amsfonts}
\usepackage{latexsym}
\usepackage{amssymb}
\usepackage{url}
\usepackage{cite}

\usepackage{algorithmic}
\usepackage{algorithm}

\DeclareMathOperator*{\grad}{grad}
\newcommand{\real}{\mathbb{R}}

\newcommand{\vtran}{\mathcal{T}}
\newcommand{\norm}[1]{\left\lVert{#1}\right\rVert}
\newcommand{\abs}[1]{\left|{#1}\right|}
\newcommand{\ip}[2]{\left\langle{#1},{#2}\right\rangle}

\newtheorem{theorem}{Theorem}[section]
\newtheorem{assumption}{Assumption}[section]
\newtheorem{definition}{Definition}[section]
\newtheorem{proposition}{Proposition}[section]
\newtheorem{problem}{Problem}[section]

\title{Sufficient Descent Riemannian Conjugate Gradient Methods \thanks{This work was supported by JSPS KAKENHI Grant Number JP18K11184.}}
\author{Hiroyuki Sakai \and Hideaki Iiduka}

\begin{document}

\maketitle

\begin{abstract}
This paper considers sufficient descent Riemannian conjugate gradient methods with line search algorithms. We propose two kinds of sufficient descent nonlinear conjugate gradient method and prove that these methods satisfy the sufficient descent condition on Riemannian manifolds. One is a hybrid method combining a Fletcher--Reeves-type method with a Polak--Ribi\`ere--Polyak--type method, and the other is a Hager--Zhang-type method, both of which are generalizations of those used in Euclidean space. Moreover, we prove that the hybrid method has a global convergence property under the strong Wolfe conditions and the Hager--Zhang-type method has the sufficient descent property regardless of whether a line search is used or not. Further, we review two kinds of line search algorithm on Riemannian manifolds and numerically compare our generalized methods by solving several Riemannian optimization problems. The results show that the performance of the proposed hybrid methods greatly depends on the type of line search used. Meanwhile, the Hager--Zhang-type method has the fast convergence property regardless of the type of line search used.
\end{abstract}

\section{Introduction}
Nonlinear conjugate gradient methods aim to solve unconstrained optimization problems in Euclidean space. Conjugate gradient methods have been developed by Hestenes and Stiefel \cite{hestenes1952methods} for solving linear systems whose coefficient matrix is symmetric positive-definite. Fletcher and Reeves \cite{fletcher1964function} extended the conjugate gradient method to unconstrained nonlinear optimization problems. Theirs is the first nonlinear conjugate gradient method in Euclidean space. Since then, various nonlinear conjugate gradient methods have been proposed (see \cite{al1985descent, polak1969note, polyak1969conjugate, dai1999nonlinear}); they have been summarized by Hager and Zhang in \cite{hager2006survey}. A sufficient descent condition is used to analyze the global convergence of conjugate gradient methods with inexact line searches. Hager and Zhang \cite{hager2005new} proposed a conjugate gradient method whose search direction satisfies the sufficient descent condition regardless of whether a line search is used or not. In addition, Dai \cite{dai2010nonlinear} proposed nonlinear conjugate gradient methods that are generalizations of the Hager--Zhang method. His method also satisfies the sufficient descent condition regardless of whether a line search is used or not. A nonlinear conjugate gradient method that satisfies the sufficient descent condition is called a sufficient descent nonlinear conjugate gradient method. Narushima and Yabe summarized the sufficient descent nonlinear conjugate gradient methods in \cite{narushima2014survey}.

The conjugate gradient method in Euclidean space can be generalized to a Riemannian manifold. In \cite{smith1994optimization}, Smith introduced the notion of Riemannian optimization. He used the exponential map and parallel transport to generalize the optimization method from Euclidean space to a Riemannian manifold. However, in general, using the exponential map or parallel transport on a Riemannian manifold is not computationally efficient. Absil, Mahony, and Sepulchre \cite{absil2008optimization} proposed to use a mapping called retraction that approximates the exponential map.
Moreover, they introduced the notion of vector transport, which approximates parallel transport.
Various methods of retraction and vector transport on Stiefel manifolds have been summarized and numerically compared by Zhu \cite{zhu2017riemannian}.

Ring and Wirth \cite{ring2012optimization} proposed a Fletcher--Reeves type of nonlinear conjugate gradient method on Riemannian manifolds with retraction and vector transport. They indicated that the Fletcher--Reeves method converges globally when each step size satisfies the strong Wolfe conditions \cite{wolfe1969convergence, wolfe1971convergence}. However, their convergence analysis assumed that vector transport satisfies the Ring-Wirth nonexpansive condition (see \eqref{eq:nonexpansive} for the definition of the Ring-Wirth nonexpansive condition). Vector transports that do not satisfy this condition have also been used (see \cite[Section 5]{sato2013new}). In \cite{sato2013new}, Sato and Iwai introduced the notion of scaled vector transport \cite[Definition 2.2]{sato2013new} to remove this impractical assumption from the convergence analysis. They proved that by using scaled vector transport, the Fletcher--Reeves method on a Riemannian manifold generates a descent direction at every iteration and converges globally without the Ring-Wirth nonexpansive condition. Similarly, Sato \cite{sato2015dai} used scaled vector transport in a convergence analysis. He indicated that the Dai--Yuan-type Riemannian conjugate gradient method generates a descent direction at every iteration and converges globally under the Wolfe conditions. In \cite{sakai2020hybrid}, Sakai and Iiduka proposed the hybrid Riemannian conjugate gradient method, which combines the Hestenes--Stiefel and Dai--Yuan methods. They proved that by using scaled vector transport, this hybrid method generates a descent direction at every iteration and converges globally under the strong Wolfe conditions.

In this paper, we focus on the sufficient descent condition\cite{narushima2014survey} and sufficient descent conjugate gradient method on Riemannian manifolds. The sufficient descent condition is stronger than the standard descent condition. We propose two kinds of sufficient descent nonlinear conjugate method for Riemannian manifolds. One is a hybrid formula combining the Fletcher-Reeves method with the Polak--Ribi\`ere--Polyak method, and we prove that, using scaled vector transport, this hybrid method has the global convergence property under the strong Wolfe conditions. The other is a formula that satisfies the sufficient descent condition regardless of whether a line search is used or not, and we prove that this method has this property even on Riemannian manifolds. This formula is a generalization of the Hager--Zhang method defined on Euclidean space. Moreover, we review two typical line search algorithms on Riemannian manifolds, i.e., the backtracking line search and line search algorithm with a \emph{zoom} phase. In numerical experiments, we compare the sufficient descent Riemannian conjugate gradient methods with the above two line search algorithms. The results show that the proposed hybrid method should use step sizes satisfying the strong Wolfe conditions, which guarantee its convergence (Theorem \ref{thm:Hybrid2convergence}).
This implies that the proposed hybrid method performs better with step sizes satisfying the strong Wolfe conditions than with step sizes satisfying the Armijo condition and that the performance of the hybrid method depends on the choice of step size.
Moreover, the results show that the benefit of the Hager--Zhang-type method is its fast convergence property regardless of the type of line search used, as promised by its sufficient descent property (Theorem \ref{thm:SDsufficient}). The main contribution of this paper is to show the fast convergence property of the sufficient descent Riemannian conjugate gradient methods regardless of the type of line search used.

This paper is organized as follows. Section \ref{sec:RCGM} reviews the Riemannian conjugate gradient methods and some useful concepts. Moreover, two Riemannian conjugate gradient methods are proposed in this section. Section \ref{sec:sufficient} proves that several Riemannian conjugate gradient methods satisfy the sufficient descent condition. Section \ref{sec:LineSearch} reviews two typical line search algorithms on Riemannian manifolds. Section \ref{sec:NE} provides the numerical experiments on several Riemannian optimization problems. Section \ref{sec:conclusion} concludes the paper.

\section{Riemannian Conjugate Gradient Methods}~\label{sec:RCGM}
A Riemannian manifold \cite{sakai1996riemannian, absil2008optimization} is a smooth manifold with an positive-definite inner product called the Riemannian metric in tangent spaces such as Euclidean space, sphere, and hyperbolic space. Let $M$ be a Riemannian manifold and $T_xM$ be a tangent space at a point $x \in M$. $\ip{\cdot}{\cdot}_x : T_xM \times T_xM \rightarrow \real$ denotes a Riemannian metric at a point $x \in M$. The Riemannian gradient of a smooth function $f:M \rightarrow \real$ at $x \in M$ is denoted by $\grad f(x)$. Let $TM := \bigcup_{x \in M}T_xM$ be the tangent bundle of $M$, and $\oplus$ be the Whitney sum (see \cite[Subchapter I.3 (p.16 (II))]{sakai1996riemannian}), defined as follows:
\begin{align*}
TM \oplus TM := \{(\xi , \eta) : \xi, \eta \in T_xM ,x \in M \}.
\end{align*}
For a smooth mapping $F:M \rightarrow N$ between two manifolds $M$ and $N$, $\mathrm{D}F(x) : T_xM \rightarrow T_{F(x)}N$ denotes the differential of $F$ at $x \in M$ (see \cite[Section 3]{absil2008optimization}). An unconstrained optimization problem on a Riemannian manifold $M$ is expressed as follows:
\begin{problem}~\label{pbl:main}
Let $f:M\rightarrow\real$ be smooth. Then, we would like to
\begin{align*}
\textrm{minimize}& \quad f(x), \\
\textrm{subject to}& \quad x \in M.
\end{align*}
\end{problem}
In order to generalize line search optimization algorithms to Riemannian manifolds, we will use the notions of retraction and vector transport, which are defined as follows:

\begin{definition}
[Retraction]~\label{def:retraction}
Any smooth map $R:TM{\rightarrow}M$ is called a retraction (see \cite[Chapter 4, Definition 4.1.1]{absil2008optimization}) on $M$ if it has the following properties.
\begin{itemize}
\item $R_x(0_x)=x$, where $0_x$ denotes the zero element of $T_xM$;
\item With the canonical identification $T_{0_x}T_xM \simeq T_xM$, $R_x$ satisfies
\begin{align*}
\mathrm{D}R_x(0_x)[\xi]=\xi
\end{align*}
for all $\xi \in T_xM$,
\end{itemize}
where $R_x$ denotes the restriction of $R$ to $T_xM$.
\end{definition}

\begin{definition}
[Vector transport]~\label{def:vector transport}
Any smooth map $\vtran:TM \oplus TM \rightarrow TM:(\eta ,\xi) \mapsto \vtran_\eta(\xi)$ is called a vector transport (see \cite[Chapter 8, Definition 8.1.1]{absil2008optimization}) on $M$ if it has the following properties.
\begin{itemize}
\item There exists a retraction $R$, called the retraction associated with $\vtran$, such that $\vtran_{\eta}(\xi) \in T_{R_x(\eta)}M$ for all $x \in M$, and for all $\eta ,\xi \in T_xM$;
\item $\vtran_{0_x}(\xi)=\xi$ for all $\xi \in T_xM$;
\item $\vtran_{\eta}(a\xi + b\zeta)=a\mathcal{T}_{\eta}(\xi)+b\vtran_{\eta}(\zeta)$ for all $a,b \in \real$, and for all $\eta ,\xi ,\zeta \in T_xM$.
\end{itemize}
\end{definition}
Retraction and vector transport are generalizations of the exponential map and parallel transport, respectively. We will use the Ring-Wirth nonexpansive condition \cite[Proposition 15]{ring2012optimization}, which is a vector transport $\vtran$ satisfying
\begin{align}~\label{eq:nonexpansive}
\norm{\vtran_\eta(\xi)}_{R_x(\eta)} \leq \norm{\xi}_x,
\end{align}
to establish global convergence for the Fletcher-Reeves type Riemannian conjugate gradient method. In this paper, we will focus on the differentiated retraction $\vtran^R$ of $R$ as a vector transport, defined by
\begin{align*}
\vtran^{R}_{\eta}(\xi):=\mathrm{D}R_{x}(\eta)[\xi],
\end{align*}
where $x \in M$ and $\eta , \xi \in T_xM$. Then, the retraction $R$ is associated with $\vtran^R$. However, the differentiated retraction $\vtran^R$ does not always satisfy the Ring-Wirth nonexpansive condition \eqref{eq:nonexpansive}.
To overcome this difficulty, Sato and Iwai \cite{sato2013new} introduced the notion of scaled vector transport. Scaled vector transport $\vtran^S$ respect to a retraction $R$ is defined for $\xi ,\eta \in T_xM$ as
\begin{align}~\label{eq:scaled}
\vtran^{S}_{\eta}(\xi):=
\begin{cases}
\vtran^R_{\eta}(\xi), & \textrm{if}\, \norm{\vtran_{\eta}^R(\xi)}_{R_{x}(\xi)}\leq\norm{\eta}_{x}, \\
\dfrac{\norm{\eta}_{x}}{\norm{\vtran^R_{\eta}(\xi)}_{R_{x}(\xi)}}\vtran^R_{\eta}(\xi), & \textrm{otherwise}.
\end{cases}
\end{align}

The general framework of Riemannian conjugate gradient methods is described in Algorithm \ref{alg:conjugate}.
\begin{algorithm}
\caption{General framework of Riemannian conjugate gradient method with scaled vector transport for solving Problem \ref{pbl:main} \cite{absil2008optimization, ring2012optimization, sato2013new, sato2015dai}. \label{alg:conjugate}}
\begin{algorithmic}[1]
\REQUIRE A Riemann manifold $M$, a retraction $R$, a smooth function $f:M\rightarrow\real$, an initial point $x_0 \in M$, convergence tolerance $\epsilon > 0$.
\ENSURE Sequence $\{x_k\}_{k=0,1,\cdots} \subset M$.
\STATE Set $\eta_0 = -g_0 := -\grad f(x_0)$
\STATE $k \leftarrow 0.$
\WHILE{$\norm{g_k}_{x_k} > \epsilon$}
\STATE Determine the positive step size $\alpha_k > 0$ and set \begin{align}~\label{eq:update} x_{k+1}=R_{x_k}(\alpha_k\eta_k). \end{align}
\STATE Compute $g_{k+1} = -\grad f(x_{k+1})$.
\STATE Compute the parameter $\beta_{k+1}$.
\STATE Set the search direction 
\begin{align}~\label{eq:direction} \eta_{k+1} = -g_{k+1} + \beta_{k+1}\vtran^S_{\alpha_k\eta_k}(\eta_k), \end{align}
where $\vtran^S$ is the scaled vector transport \eqref{eq:scaled} with respect to $R$.
\STATE $k \leftarrow k + 1.$
\ENDWHILE
\end{algorithmic}
\end{algorithm}

In this paper, we say that the search direction $\eta_k \in T_{x_k}M$ is a descent direction if $\ip{g_k}{\eta_k} < 0$ holds. In addition, $\eta_k$ is a sufficient descent direction (see \cite{narushima2014survey}) if the sufficient descent condition, 
\begin{align}~\label{eq:sufficient}
\ip{g_k}{\eta_k} \leq -\kappa\norm{g_k}^2,
\end{align}
holds for some constant $\kappa > 0$. In \eqref{eq:update}, for a given descent direction $\eta_k \in T_xM$ at $x \in M$, one often chooses a step size $\alpha_k > 0$ to satisfy the Armijo condition \cite[Definition 2.3]{hosseini2018line}, \cite[(1a)]{ring2012optimization}, namely,
\begin{align}~\label{eq:Armijo}
f(R_{x_k}(\alpha_k\eta_k)) \leq f(x_k)+c_1\alpha_k\ip{\grad f(x_k)}{\eta_k}_{x_k},
\end{align}
where $0 < c_1 < 1$. The following condition is called the curvature condition \cite[Definition 2.5]{hosseini2018line}:
\begin{align}
\label{eq:Wolfe}
\ip{\grad f(R_{x_k}(\alpha_k\eta_k))} {\vtran^R_{\alpha_k\eta_k}(\eta_k)}_{R_{x_k}(\alpha_k\eta_k)} \geq c_2 \ip{\grad f(x_k)}{\eta_k}_{x_k},
\end{align}
where $0<c_1<c_2<1$. Conditions \eqref{eq:Armijo} and \eqref{eq:Wolfe} are called the Wolfe conditions \cite[Definition 2.7]{hosseini2018line}, \cite[(1a), (1b)]{ring2012optimization}. If condition \eqref{eq:Wolfe} is replaced by
\begin{align}
\label{eq:sWolfe}
\abs{ \ip{\grad f(R_{x_k}(\alpha_k\eta_k))} {\vtran^R_{\alpha_k\eta_k}(\eta_k)}_{R_{x_k}(\alpha_k\eta_k)}} \leq c_2\abs{\ip{\grad f(x_k)}{\eta_k}_{x_k}},
\end{align}
then \eqref{eq:Armijo} and \eqref{eq:sWolfe} are called the \emph{strong} Wolfe conditions \cite[(1a), (2)]{ring2012optimization}.

In \eqref{eq:direction}, $\beta_{k+1}$ is given by generalizations of the formulas in Euclidean space (see \cite{hestenes1952methods, fletcher1964function, polak1969note, polyak1969conjugate, dai1999nonlinear}), e.g.,
\begin{align}
\label{eq:HS}
\beta_{k+1}^\mathrm{HS}&=\frac{\ip{g_{k+1}}{y_{k+1}}_{x_{k+1}}}{\ip{g_{k+1}}{\vtran^S_{\alpha_{k}\eta_{k}}(\eta_{k})}_{x_{k+1}}-\ip{g_{k}}{\eta_{k}}_{x_{k}}}, \\
\label{eq:FR}
\beta_{k+1}^\mathrm{FR}&=\frac{\norm{g_{k+1}}^2_{x_{k+1}}}{\norm{g_k}^2_{x_{k}}}, \\
\label{eq:PRP}
\beta_{k+1}^\mathrm{PRP}&=\frac{\ip{g_{k+1}}{y_{k+1}}_{x_{k+1}}}{\norm{g_k}^2_{x_{k}}}, \\
\label{eq:DY}
\beta_{k+1}^\mathrm{DY}&=\frac{\norm{g_{k+1}}^2_{x_{k+1}}}{\ip{g_{k+1}}{\vtran^S_{\alpha_{k}\eta_{k}}(\eta_{k})}_{x_{k+1}}-\ip{g_{k}}{\eta_{k}}_{x_{k}}},
\end{align}
where $y_{k+1} := g_{k+1} - \vtran^S_{\alpha_k\eta_k}(g_k)$. Formulas \eqref{eq:HS}, \eqref{eq:FR}, \eqref{eq:PRP}, and \eqref{eq:DY} are called the Hestenes--Stiefel (HS), Fletcher--Reeves (FR), Polak--Ribi\`ere--Polyak (PRP), and Dai--Yuan (DY) formulas, respectively. In \cite{sato2013new}, Sato and Iwai indicated that, by using scaled vector transport, the FR method converges globally under the \emph{strong} Wolfe conditions \eqref{eq:Armijo} and \eqref{eq:sWolfe}. In \cite{sato2015dai}, Sato proved that the DY method converges globally under the Wolfe conditions \eqref{eq:Armijo} and \eqref{eq:Wolfe}. The HS and PRP methods have good numerical performance; however, no convergence analyses have been presented for them on Rimannian manifolds. To make up for these shortcomings, hybrid-type formulas, such as
\begin{align}
\label{eq:Hybrid1}
\beta_{k+1}^\mathrm{Hyb1} &= \max\{0,\min\{\beta_{k+1}^\mathrm{HS},\beta_{k+1}^\mathrm{DY}\}\}, \\
\label{eq:Hybrid2}
\beta_{k+1}^\mathrm{Hyb2} &= \max\{0,\min\{\beta_{k+1}^\mathrm{FR},\beta_{k+1}^\mathrm{PRP}\}\},
\end{align}
have been developed in Euclidean space (see \cite{hu1991global, dai2001efficient}). Below, we call the hybrid methods using \eqref{eq:Hybrid1} and \eqref{eq:Hybrid2}, Hybrid1 and Hybrid2, respectively. The Hybrid1 method was proposed by Dai and Yuan \cite{dai2001efficient}, and the Hybrid2 method was suggested by Hu and Storey \cite{hu1991global}. In \cite{sakai2020hybrid}, Sakai and Iiduka generalized the Hybrid1 method on Riemannian manifolds and proved that it converges globally under the strong Wolfe conditions \eqref{eq:Armijo} and \eqref{eq:sWolfe}. They also showed that the numerical performance of the Hybrid1 method is better than that of the PRP method \cite[Section 4]{sakai2020hybrid}. In the next section (Theorem \ref{thm:Hybrid2sufficient}), we generalize the Hybrid2 method to Riemannian manifold and prove that it satisfies the sufficient descent condition under the strong Wolfe conditions. Moreover, we give its convergence analysis (Theorem \ref{thm:Hybrid2convergence}).

We consider the nonlinear conjugate gradient methods that can guarantee the sufficient descent condition \eqref{eq:sufficient} regardless of the type of line search used.
We generalize the Hager--Zhang (HZ) method \cite{hager2005new, hager2006survey} to Riemannian manifolds, as follows,
\begin{align}~\label{eq:HZ}
\beta_{k+1}^\mathrm{HZ} = \beta^\mathrm{HS}_{k+1} - \mu\frac{\norm{y_{k+1}}_{x_{k+1}}^2\ip{g_{k+1}}{\vtran^S_{\alpha_k\eta_k}(\eta_k)}_{x_{k+1}}}{\left(\ip{g_{k+1}}{\vtran^S_{\alpha_{k}\eta_{k}}(\eta_{k})}_{x_{k+1}}-\ip{g_{k}}{\eta_{k}}_{x_{k}}\right)^2},
\end{align}
where $y_{k+1} := g_{k+1} - \vtran^S_{\alpha_k\eta_k}(g_k)$ and $\mu > 1/4$. Moreover, we modify $\beta_{k+1}$ of the form $\beta_{k+1} = \ip{g_{k+1}}{\xi_{k+1}}_{x_{k+1}}$ to
\begin{align}~\label{eq:SD}
\beta_{k+1}^\mathrm{SD} = \beta_{k+1} - \mu \norm{\xi_{k+1}}_{x_{k+1}}^2\ip{g_{k+1}}{\vtran^S_{\alpha_{k}\eta_{k}}(\eta_{k})}_{x_{k+1}},
\end{align}
where $\xi_{k+1} \in T_{x_{k+1}}M$ is any tangent vector, $\mu > 1/4$ (see \cite{dai2010nonlinear, narushima2014survey}), and SD stands for sufficient descent.
For instance, if we set
\begin{align*}
\xi_{k+1} = \frac{y_{k+1}}{\ip{g_{k+1}}{\vtran^S_{\alpha_{k}\eta_{k}}(\eta_{k})}_{x_{k+1}}-\ip{g_{k}}{\eta_{k}}_{x_{k}}},
\end{align*}
we have $\beta_{k+1}^\mathrm{SD} = \beta_{k+1}^\mathrm{HZ}$. We will show that the SD method always satisfies the sufficient descent condition \eqref{eq:sufficient} with $\kappa = 1 - (1/4\mu)$ (Theorem \ref{thm:SDsufficient}).

\section{Sufficient Descent Properties of the Riemannian Conjugate Gradient Methods}~\label{sec:sufficient}
In this section, we recall the properties of the FR \eqref{eq:FR}, DY \eqref{eq:DY} and Hybrid1 \eqref{eq:Hybrid1} methods (see \cite{sato2013new, sato2015dai, sakai2020hybrid}).

\begin{proposition}~\label{prop:sufficient}
The following statements hold:
\begin{description}
\item[(P1)] If $\beta_{k+1} = \beta_{k+1}^\mathrm{FR}$ and $\alpha_k$ satisfies the strong Wolfe conditions \eqref{eq:Armijo} and \eqref{eq:sWolfe} with $0<c_1<c_2<1/2$, then
\begin{align*}
-\frac{1}{1-c_2}\norm{g_k}_{x_k}^2 \leq \ip{g_k}{\eta_k}_{x_k} \leq -\frac{1-2c_2}{1-c_2}\norm{g_k}_{x_k}^2,
\end{align*}
for all $k=0,1,\cdots$. Thus, the FR method satisfies the sufficient descent condition \eqref{eq:sufficient} with $\kappa = (1-2c_2)/(1-c_2) > 0$.
\item[(P2)] If $\beta_{k+1} = \beta_{k+1}^\mathrm{DY}$ and $\alpha_k$ satisfies the Wolfe conditions \eqref{eq:Armijo} and \eqref{eq:Wolfe}, then
\begin{align*}
-\frac{1}{1-c_2}\norm{g_k}_{x_k}^2 \leq \ip{g_k}{\eta_k}_{x_k} \leq -\frac{1}{1+c_2}\norm{g_k}_{x_k}^2,
\end{align*}
for all $k=0,1,\cdots$. Thus, the DY method satisfies the sufficient descent condition \eqref{eq:sufficient} with $\kappa = 1/(1+c_2) > 0$.
\item[(P3)] If $\beta_{k+1} = \beta_{k+1}^\mathrm{Hyb1}$ and $\alpha_k$ satisfies the strong Wolfe conditions \eqref{eq:Armijo} and \eqref{eq:sWolfe}, then
\begin{align*}
-\frac{1 + c_2}{1-c_2}\norm{g_k}_{x_k}^2 \leq \ip{g_k}{\eta_k}_{x_k} \leq -\frac{1 - c_2}{1+c_2}\norm{g_k}_{x_k}^2,
\end{align*}
for all $k=0,1,\cdots$. Thus, the Hybrid1 method satisfies the sufficient descent condition \eqref{eq:sufficient} with $\kappa = (1-c_2)/(1+c_2) > 0$.
\end{description}
\end{proposition}
Proposition \ref{prop:sufficient} implies that whether Algorithm \ref{alg:conjugate} using FR, DY or Hybrid1 satisfies the sufficient descent condition \eqref{eq:sufficient} depends on not only the parameter $\beta_{k+1}$ methods used but also the line search, in the sense that the line search has to impose the strong Wolfe conditions \eqref{eq:Armijo} and \eqref{eq:sWolfe}.
Here, (P1) is the result in \cite[Lemma 4.1]{sato2013new}, and (P2) and (P3) are easily shown from \cite[(35)]{sakai2020hybrid}.

\subsection{A Sufficient Descent Property of the Hybrid2 method}
In this section, we show that the Hybrid2 method generates a sufficient descent direction \eqref{eq:sufficient} at every iteration. This result is a simple extension of Proposition \ref{prop:sufficient} (P1).

\begin{theorem}~\label{thm:Hybrid2sufficient}
Let $f:M \rightarrow \real$ be a smooth function. If each $\alpha_k>0$ satisfies the strong Wolfe conditions \eqref{eq:Armijo} and \eqref{eq:sWolfe}, with $0 < c_1 < c_2 < 1/2$, and $\beta_{k+1}$ satisfies\footnote{The formulas defined by \eqref{eq:FR} and \eqref{eq:Hybrid2} satisfy $\abs{\beta_{k+1}} \leq \beta_{k+1}^\mathrm{FR}$.} $\abs{\beta_{k+1}} \leq \beta_{k+1}^\mathrm{FR}$, then any sequence $\{x_k\}_{k=0,1\cdots}$ generated by Algorithm \ref{alg:conjugate} satisfies
\begin{align}~\label{eq:Hybrid2descent}
-\frac{1}{1-c_2}\norm{g_k}_{x_k}^2 \leq \ip{g_k}{\eta_k}_{x_k} \leq -\frac{1-2c_2}{1-c_2}\norm{g_k}_{x_k}^2,
\end{align}
for all $k=0,1, \cdots$.
\end{theorem}
\begin{proof}
The proof is by induction. If $k=0$, \eqref{eq:Hybrid2descent} clearly holds. Assume that \eqref{eq:Hybrid2descent} holds for some $k \geq 0$. By $c_2 < 1/2$, we obtain $\ip{g_k}{\eta_k}_{x_k} < 0$. From the search direction \eqref{eq:direction}, we have
\begin{align*}
\frac{\ip{g_{k+1}}{\eta_{k+1}}_{x_{k+1}}}{\norm{g_{k+1}}_{x_{k+1}}^2} = -1 + \beta_{k+1}\frac{\ip{g_{k+1}}{\vtran^S_{\alpha_k\eta_k}(\eta_k)}_{x_{k+1}}}{\norm{g_{k+1}}_{x_{k+1}}^2},
\end{align*}
which implies
\begin{align}~\label{eq:Hybrid2induction}
\frac{\ip{g_{k+1}}{\eta_{k+1}}_{x_{k+1}}}{\norm{g_{k+1}}_{x_{k+1}}^2} = -1 + \frac{\beta_{k+1}}{\beta^\mathrm{FR}_{k+1}}\frac{s_k\ip{g_{k+1}}{\vtran^R_{\alpha_k\eta_k}(\eta_k)}_{x_{k+1}}}{\norm{g_{k}}_{x_{k}}^2},
\end{align}
where
\begin{align*}
s_k:=\min\left\{1,\frac{\norm{\eta_k}_{x_k}}{\norm{\vtran^R_{\alpha_k\eta_k}(\eta_k)}_{x_{k+1}}}\right\} \in [0,1].
\end{align*}
From the second condition of the strong Wolfe conditions \eqref{eq:sWolfe} and $\ip{g_k}{\eta_k}_{x_k}<0$, we obtain
\begin{align*}
\abs{\beta_{k+1}\ip{g_{k+1}}{\vtran^R_{\alpha_k\eta_k}(\eta_k)}_{x_{k+1}}} \leq -c_2\abs{\beta_{k+1}}\ip{g_k}{\eta_k}_{x_k},
\end{align*}
which together with \eqref{eq:Hybrid2induction} implies
\begin{align*}
-1 + c_2s_k \frac{\abs{\beta_{k+1}}}{\beta_{k+1}^\mathrm{FR}}\frac{\ip{g_k}{\eta_k}_{x_k}}{\norm{g_k}_{x_k}^2}
\leq \frac{\ip{g_{k+1}}{\eta_{k+1}}_{x_{k+1}}}{\norm{g_{k+1}}_{x_{k+1}}^2}
\leq -1 - c_2s_k \frac{\abs{\beta_{k+1}}}{\beta_{k+1}^\mathrm{FR}}\frac{\ip{g_k}{\eta_k}_{x_k}}{\norm{g_k}_{x_k}^2}.
\end{align*}
From the left-hand side of the induction hypothesis \eqref{eq:Hybrid2descent}, we have
\begin{align*}
-1 - c_2s_k \frac{\abs{\beta_{k+1}}}{\beta_{k+1}^\mathrm{FR}}\frac{1}{1-c_2}
\leq \frac{\ip{g_{k+1}}{\eta_{k+1}}_{x_{k+1}}}{\norm{g_{k+1}}_{x_{k+1}}^2}
\leq -1 + c_2s_k \frac{\abs{\beta_{k+1}}}{\beta_{k+1}^\mathrm{FR}}\frac{1}{1-c_2}.
\end{align*}
Utilizing the assumption $\abs{\beta_{k+1}} \leq \beta_{k+1}^\mathrm{FR}$ and $0 \leq s_k \leq 1$, we obtain
\begin{align*}
-1 - \frac{c_2}{1-c_2}
\leq \frac{\ip{g_{k+1}}{\eta_{k+1}}_{x_{k+1}}}{\norm{g_{k+1}}_{x_{k+1}}^2}
\leq -1 + \frac{c_2}{1-c_2}.
\end{align*}
This implies that \eqref{eq:Hybrid2descent} holds for $k+1$. \qed
\end{proof}

Moreover, we prove the global convergence of the Hybrid2 method under the strong Wolfe conditions and the following assumption.
\begin{assumption}~\label{asm:main}
Let $M$ be a Riemannian manifold and R be a retraction on M. Let $f:M\rightarrow\real$ be a smooth, bounded below function. Then, we assume that there exists $L>0$ such that
\begin{align*}
|\mathrm{D}(f \circ R_x)(t\eta)[\eta]-\mathrm{D}(f \circ R_x)(0_x)[\eta]| \leq Lt,
\end{align*}
where $x \in M$, $\eta \in T_xM$, $\norm{\eta}_x=1$ and $t \geq 0$.
\end{assumption}
This is the assumption for Zoutendijk's theorem (Theorem \ref{thm:Zoutendijk}) on Riemannian manifolds. Zoutendijk's theorem on Riemannian manifolds is as follows:
\begin{theorem}[Zoutendijk \cite{sato2013new}]~\label{thm:Zoutendijk}
Let $(M,g)$ be a Riemannian manifold and R be a retraction on M. Suppose $f:M\rightarrow\real$ satisfies Assumption \ref{asm:main}. Suppose further that in Algorithm \ref{alg:conjugate}, each step size $\alpha_k>0$ satisfies the strong Wolfe conditions \eqref{eq:Armijo} and \eqref{eq:sWolfe}. Then the following series converges:
\begin{align}
\label{eq:Zoutendijk}
\sum_{k=0}^\infty \frac{\ip{g_k}{\eta_k}_{x_k}^2}{\norm{\eta_k}^2_{x_k}} < \infty .
\end{align}
\end{theorem}

The proof of this theorem is along the lines of Zoutendijk's theorem in Euclidean space (see \cite[Theorem 3.3]{ring2012optimization}). Global convergence proofs for Riemannian conjugate gradient methods are often based on Zoutendijk's theorem. Theorem \ref{thm:Hybrid2convergence} guarantees global convergence of the Hybrid2 method \eqref{eq:Hybrid2}. It is a generalization of the convergence theorem of the Hybrid2 method in Euclidean space \cite{gilbert1992global}.

\begin{theorem}~\label{thm:Hybrid2convergence}
Let $f:M \rightarrow \real$ be a function satisfying Assumption \ref{asm:main}. If each $\alpha_k>0$ satisfies the strong Wolfe conditions \eqref{eq:Armijo} and \eqref{eq:sWolfe}, with $0 < c_1 < c_2 < 1/2$, and $\beta_{k+1}$ satisfies $\abs{\beta_{k+1}} \leq \beta_{k+1}^\mathrm{FR}$, then any sequence $\{x_k\}_{k=0,1\cdots}$ generated by Algorithm \ref{alg:conjugate} satisfies
\begin{align}~\label{eq:Hybrid2converge}
\liminf_{k \to \infty}\norm{g_k}_{x_k}=0.
\end{align}
\end{theorem}
\begin{proof}
We prove \eqref{eq:Hybrid2converge} by contradiction. If $g_{k_0} = 0$ for some $k_0$, then \eqref{eq:Hybrid2converge} follows. Assume that
\begin{align*}
\liminf_{k \to \infty}\norm{g_k}_{x_k} > 0.
\end{align*}
Then, noting $\norm{g_k}_{x_k}\neq 0$ for all $k$, there exists $\gamma > 0$ such that
\begin{align*}
\norm{g_k}_{x_k} \geq \gamma > 0,
\end{align*}
for all $k$. From \eqref{eq:sWolfe} and \eqref{eq:Hybrid2descent}, we have
\begin{align*}
\abs{\ip{g_k}{\vtran^R_{\alpha_{k-1}\eta_{k-1}}(\eta_{k-1})}_{x_k}} &\leq -c_2\ip{g_{k-1}}{\eta_{k-1}}_{x_{k-1}} \\
&\leq \frac{c_2}{1-c_2}\norm{g_{k-1}}_{x_{k-1}}^2.
\end{align*}
Thus, from \eqref{eq:direction} and \eqref{eq:Hybrid2descent}, and using the condition $\abs{\beta_k} \leq \beta^\mathrm{FR}_k = \norm{g_k}_{x_k}^2 / \norm{g_{k-1}}_{x_{k-1}}^2$, we have
\begin{align*}
\norm{\eta_k}_{x_k}^2 &\leq \norm{g_k}_{x_k}^2 + 2s_k\abs{\beta_k\ip{g_k}{\vtran^R_{\alpha_{k-1}\eta_{k-1}}(\eta_{k-1})}_{x_k}} + \norm{\beta_k\vtran^S_{\alpha_{k-1}\eta_{k-1}}(\eta_{k-1})}^2_{x_k} \\
&\leq \norm{g_k}_{x_k}^2 + \frac{2c_2}{1-c_2}\abs{\beta_k}\norm{g_{k-1}}^2_{x_{k-1}} + \beta_k^2\norm{\eta_{k-1}}_{x_{k-1}}^2 \\
&\leq \hat{c}\norm{g_k}_{x_k}^2 + \beta_k^2\norm{\eta_{k-1}}_{x_{k-1}}^2,
\end{align*}
where $\hat{c} := (1 + c_2) / (1 - c_2) > 1$. Applying this equation repeatedly, we obtain
\begin{align*}
\norm{\eta_k}_{x_k}^2 &\leq \hat{c}\norm{g_k}_{x_k}^2 + \beta_k^2\left(\hat{c}\norm{g_{k-1}}_{x_{k-1}}^2 + \beta_{k-1}^2\norm{\eta_{k-2}}_{x_{k-2}}^2\right) \\
&\leq \hat{c}\left(\norm{g_k}_{x_k}^2 + \beta_k^2\norm{g_{k-1}}_{x_{k-1}}^2 + \cdots + \beta_k^2\beta_{k-1}^2\cdots\beta_2^2\norm{g_1}_{x_1}^2\right) + \beta_k^2\beta_{k-1}^2\cdots\beta_1^2\norm{\eta_0}_{x_0}^2 \\
&\leq \hat{c}\norm{g_k}_{x_k}^4\left(\frac{1}{\norm{g_k}_{x_k}^2} + \frac{1}{\norm{g_{k-1}}_{x_{k-1}}^2} + \cdots + \frac{1}{\norm{g_1}_{x_1}^2}\right) + \frac{\norm{g_k}_{x_k}^4}{\norm{g_0}_{x_0}^2} \\
&< \hat{c}\norm{g_k}_{x_k}^4\sum_{j=0}^k\frac{1}{\norm{g_j}_{x_j}^2} \\
&\leq \frac{\hat{c}}{\gamma^2}\norm{g_k}_{x_k}^4(k+1).
\end{align*}
This implies that
\begin{align*}
\frac{\norm{g_k}_{x_k}^4}{\norm{\eta_k}_{x_k}^2} \geq \frac{\gamma^2}{\hat{c}(k+1)},
\end{align*}
which together with \eqref{eq:Hybrid2descent}, gives
\begin{align*}
\sum_{k=0}^\infty \frac{\ip{g_k}{\eta_k}_{x_k}^2}{\norm{\eta_k}^2_{x_k}} &= \sum_{k=0}^\infty \frac{\norm{g_k}_{x_k}^4}{\norm{\eta_k}_{x_k}^2}\frac{\ip{g_k}{\eta_k}_{x_k}^2}{\norm{g_k}_{x_k}^4} \\
&\geq \left( \frac{2c_2 - 1}{1 -c_1} \right)^2 \sum_{k=0}^\infty \frac{\gamma^2}{\hat{c}(k+1)} \\
& = \infty.
\end{align*}
This contradicts \eqref{eq:Zoutendijk} in Zoutendijk's theorem (Theorem \ref{thm:Zoutendijk}) and completes the proof. \qed
\end{proof}

\subsection{Sufficient Descent Property of the SD method}
Theorem \ref{thm:SDsufficient} asserts that the SD method \eqref{eq:SD} produces sufficient descent directions \eqref{eq:sufficient} regardless of the choice of the step size $\alpha_k$.
\begin{theorem}~\label{thm:SDsufficient}
Let $f:M \rightarrow \real$ be a smooth function. If $\beta_{k+1} = \beta_{k+1}^\mathrm{SD}$, then any sequence $\{x_k\}_{k=0,1\cdots}$ generated by Algorithm \ref{alg:conjugate} satisfies
\begin{align}~\label{eq:SDsufficient}
\ip{g_k}{\eta_k}_{x_k} \leq -\left(1-\frac{1}{4\mu}\right)\norm{g_{k}}_{x_k}^2.
\end{align}
\end{theorem}
\begin{proof}
From \eqref{eq:SD}, we obtain
\begin{align}~\label{eq:HZtemp}
\begin{split}
\ip{g_k}{\eta_k}_{x_k} = &-\norm{g_k}_{x_k}^2 + \ip{g_k}{\xi_k}_{x_k}\ip{g_k}{\vtran^S_{\alpha_{k-1}\eta_{k-1}}(\eta_{k-1})}_{x_k} \\
& \quad - \mu\norm{\xi_k}_{x_k}^2\ip{g_k}{\vtran^S_{\alpha_{k-1}\eta_{k-1}}(\eta_{k-1})}_{x_k}^2.
\end{split}
\end{align}
An upper bound for the middle term in \eqref{eq:HZtemp} is obtained using the inequality,
\begin{align*}
\ip{u_k}{v_k}_{x_k} \leq \frac{\norm{u_k}_{x_k}^2 + \norm{v_k}_{x_k}^2}{2}
\end{align*}
with the choice.
\begin{align*}
u_k := \frac{1}{\sqrt{2\mu}}g_k \quad and \quad v_k := \sqrt{2\mu}\ip{g_k}{\vtran^S_{\alpha_{k-1}\eta_{k-1}}(\eta_{k-1})}\xi_k.
\end{align*}
Then, we have
\begin{align*}
&\ip{g_k}{\xi_k}_{x_k}\ip{g_k}{\vtran^S_{\alpha_{k-1}\eta_{k-1}}(\eta_{k-1})}_{x_k} \\
&\qquad \leq \frac{1}{4\mu}\norm{g_k}_{x_k}^2 + \mu\norm{\xi_k}_{x_k}^2\ip{g_k}{\vtran^S_{\alpha_{k-1}\eta_{k-1}}(\eta_{k-1})}_{x_k}^2.
\end{align*}
Combining this with \eqref{eq:HZtemp}, we obtain \eqref{eq:SDsufficient}. \qed
\end{proof}

\section{Line Search Algorithm on Riemannian Manifolds}~\label{sec:LineSearch}
In this section, we review two line search algorithms on Riemannian manifolds. In Algorithm \ref{alg:conjugate}, we need to use a line search algorithm to determine the step size $\alpha_k$. A backtracking line search algorithm is widely used in optimization algorithms in Euclidean space (see \cite[Chapter 3, Algorithm 3.1]{nocedal2006numerical}) and on Riemannian manifolds \cite{absil2008optimization} to find a step size that satisfies the Armijo condition \eqref{eq:Armijo}. Algorithm \ref{alg:BacktrackArmijo} is a backtracking line search on Riemannian manifolds \cite[Algorithm 1]{absil2008optimization}. This algorithm multiplies a positive constant $\rho > 0$ until a step size $\alpha$ satisfying the Armijo condition is found.

\begin{algorithm}
\caption{Backtracking line search on Riemannian manifold $M$ \cite[Algorithm 1]{absil2008optimization}. \label{alg:BacktrackArmijo}}
\begin{algorithmic}[1]
\REQUIRE A smooth function $f:M\rightarrow\real$, a point $x \in M$, a descent direction $\eta \in T_xM$, scalars $0 < \alpha_{\mathrm{hi}}$, $\rho \in (0,1)$.
\ENSURE A positive step size $\alpha > 0$ satisfying the Armijo condition \eqref{eq:Armijo}.
\STATE $\alpha \leftarrow \alpha_{\mathrm{hi}}$
\WHILE{$f(R_{x}(\alpha\eta)) > f(x) + c_1\alpha \ip{\grad f(x)}{\eta}_{x}$}
\STATE $\alpha \leftarrow \rho\alpha$
\ENDWHILE
\RETURN $\alpha$
\end{algorithmic}
\end{algorithm}

However, a backtracking line search algorithm cannot be used for the Wolfe or the strong Wolf conditions. To find a step size satisfying the \emph{strong} Wolfe conditions, In \cite[Section 5.1]{sato2015dai}, Sato presented Algorithm \ref{alg:LineSearch}, a generalization of the algorithm in \cite[Chapter 3, Algorithm 3.5]{nocedal2006numerical} for strong Wolfe conditions in Euclidean space. Algorithm \ref{alg:LineSearch} calls the \emph{zoom} function (Algorithm \ref{alg:Zoom}), which successively decreases the size of the interval until an acceptable step size is found (see \cite[Chapter 3, Algorithm 3.6]{nocedal2006numerical}). The parameter $\alpha_\mathrm{hi}$ is a user-supplied bound on the maximum step size. Algorithm \ref{alg:LineSearch} returns a positive step size, $\alpha_\star > 0$, that satisfies the strong Wolfe conditions. If we find a step size satisfying the Wolfe conditions, we replace the condition of step 6 of the Algorithms \ref{alg:LineSearch} and \ref{alg:Zoom} with $\phi^\prime(\alpha_i) \geq c_2\phi^\prime(0)$ (see \cite[Section 5.1]{sato2015dai}).

\begin{algorithm}
\caption{Line search algorithm on Riemannian manifold $M$ \cite[Section 5.1]{sato2015dai}. \label{alg:LineSearch}}
\begin{algorithmic}[1]
\REQUIRE A smooth function $f:M\rightarrow\real$, a point $x \in M$, a descent direction $\eta \in T_xM$, scalars $0<c_1<c_2<1$, $0 < \alpha_{\mathrm{hi}}$ and $\alpha_0 \in (0,\alpha_\mathrm{hi})$.
\ENSURE A positive step size $\alpha > 0$ satisfying the \emph{strong} Wolfe conditions \eqref{eq:Armijo} and \eqref{eq:sWolfe}.
\STATE Set $\phi(\alpha) = f(R_x(\alpha\eta))$.
\STATE $i \leftarrow 0$.
\LOOP
\IF{$\phi(\alpha_i) > \alpha(0) + \alpha c_1\phi^\prime(0)$ \OR $[\phi(\alpha_i) \geq \phi(\alpha_{i-1})$ \AND $i \geq 1]$ }
\STATE Set $\alpha_\star = \mathrm{Zoom}(\alpha_{i-1}, \alpha_i)$ and stop.
\ELSIF{$\abs{\phi^\prime(\alpha_i)} \leq -c_2\phi^\prime(0)$}
\STATE Set $\alpha_\star = \alpha_i$ and stop.
\ELSIF{$\phi^\prime(0) \geq 0$}
\STATE Set $\alpha_\star = \mathrm{Zoom}(\alpha_i, \alpha_{i-1})$ and stop.
\ENDIF
\STATE Choose $\alpha_{i+1} \in (\alpha_i, \alpha_\mathrm{hi})$.
\STATE $i \leftarrow i+1$.
\ENDLOOP
\RETURN $\alpha_\star$
\end{algorithmic}
\end{algorithm}

\begin{algorithm}
\caption{Zoom \cite[Section 5.1]{sato2015dai}, \cite[Chapter 3, Algorithm 3.6]{nocedal2006numerical} \label{alg:Zoom} }
\begin{algorithmic}[1]
\REQUIRE Scalars $\alpha_\mathrm{min}, \alpha_\mathrm{max} > 0$, and $\phi(\alpha) = f(R_x(\alpha\eta))$.
\ENSURE $\alpha = \mathrm{Zoom}(\alpha_\mathrm{min}, \alpha_\mathrm{max})$.
\LOOP
\STATE Interpolate (using quadratic, cubic, or bisection) to find a trial step length $\alpha_j \in (\alpha_\mathrm{min}, \alpha_\mathrm{max})$.
\IF{$\phi(\alpha_j) > \phi(0) + c_1\alpha_j\phi^\prime(0)$ \OR $\phi(\alpha_j) \geq \phi(\alpha_\mathrm{min})$}
\STATE $\alpha_\mathrm{max} \leftarrow \alpha_j$
\ELSE
\IF{$\abs{\phi^\prime(\alpha_j)} \leq -c_2\phi^\prime(0)$}
\STATE Set $\alpha_\star = \alpha_j$ and stop.
\ELSIF{$\phi^\prime(\alpha_j) (\alpha_\mathrm{max} - \alpha_\mathrm{min}) \geq 0$}
\STATE $\alpha_\mathrm{max} \leftarrow \alpha_\mathrm{min}$.
\ENDIF
\STATE $\alpha_\mathrm{min} \leftarrow \alpha_j$.
\ENDIF
\ENDLOOP
\RETURN $\alpha_\star$
\end{algorithmic}
\end{algorithm}

\section{Numerical Experiments}~\label{sec:NE}
Our experiments used source code based on \texttt{pymanopt}\footnote{\url{https://www.pymanopt.org/}} (see \cite{townsend2016pymanopt}).
In addition, Algorithm \ref{alg:LineSearch} was based on an implementation by SciPy\footnote{\url{https://docs.scipy.org/doc/scipy/reference/}} in Euclidean space.
Python implementations of the methods used in the numerical experiments are available at \url{https://github.com/iiduka-researches/202104-sufficient}.
We solved four different Riemannian optimization problems (Problems \ref{pbl:rayleigh}--\ref{pbl:off-diag}).

Problem \ref{pbl:rayleigh} is the Rayleigh-quotient minimization problem on the unit sphere (see \cite[Chapter 4.6]{absil2008optimization}).
\begin{problem}~\label{pbl:rayleigh}
For $A \in \mathcal{S}^n_{++}$,
\begin{align*}
\textrm{minimize}& \quad f(x)=x^\top Ax, \\
\textrm{subject to}& \quad x \in \mathbb{S}^{n-1} := \{ x \in \real^{n} : \norm{x} = 1\},
\end{align*}
where $\norm{\cdot}$ denotes the Euclidean norm and $\mathcal{S}^n_{++}$ denotes the set of all $n \times n$ symmetric positive-definite matrices.
\end{problem}
In the experiments, we set $n = 100$ and generated a matrix $A \in \mathcal{S}^n_{++}$ with randomly chosen elements by using \texttt{sklearn.datasets.make\_spd\_matrix}.

Problem \ref{pbl:brockett} is the Brockett-cost-function minimization problem on a Stiefel manifold (see \cite[Chapter 4.8]{absil2008optimization}).
\begin{problem}~\label{pbl:brockett}
For $A \in \mathcal{S}^n_{++}$ and $N = \mathrm{diag}(\mu_0, \cdots ,\mu_p)$ $(0 \leq \mu_0\leq \cdots \leq\mu_p)$,
\begin{align*}
\textrm{minimize}& \quad f(X)=\mathrm{tr}(X^\top AXN)\\
\textrm{subject to}& \quad X \in \mathrm{St}(p,n) := \{X \in \real^{n \times p} : X^\top X=I_p\}.
\end{align*}
\end{problem}
In the experiments, we set $p=5$, $n=20$ and $N := \mathrm{diag}(1, 2, 3, 4, 5)$ and generated a matrix $A \in \mathcal{S}^n_{++}$ with randomly chosen elements by using \texttt{sklearn.datasets.make\_spd\_matrix}.

In \cite{vandereycken2013low}, Vandereycken discussed the following robust matrix completion problem (Problem \ref{pbl:completion}).
\begin{problem}~\label{pbl:completion}
For $A \in \real^{m \times n}$, and a subset $\Omega$ of the complete set of entries $\{1, \cdots ,m\}\times\{1, \cdots ,n\}$,
\begin{align*}
\textrm{minimize}& \quad f(X)=\norm{P_\Omega(X - A)}^2_F, \\
\textrm{subject to}& \quad X \in M_k := \{ X \in \real^{m \times n} : \mathrm{rank}(X) = k\},
\end{align*}
where $\norm{\cdot}_F$ denotes the Frobenius norm and
\begin{align*}
P_\Omega : \real^{m \times n} \rightarrow \real^{m \times n}, X_{ij} \mapsto
\begin{cases}
X_{ij} & (i,j) \in \Omega \\
0 & (i,j) \not\in \Omega
\end{cases}.
\end{align*}
\end{problem}
In the experiments, we set $m=n=100$ and $k=4$, and $\Omega$ contained each pair $(i,j) \in \{1, \cdots ,m\}\times\{1, \cdots ,n\}$ with probability $1/2$. Moreover, we used a matrix $A \in \real^{m \times n}$ that was generated with randomly chosen elements by using \texttt{numpy.random.randn}.

In \cite{absil2006joint}, Absil and Gallivan introduced the following off-diagonal cost function minimization problem on oblique manifolds (Problem \ref{pbl:off-diag}).
\begin{problem}~\label{pbl:off-diag}
For $C_i \in \mathcal{S}^{n}$ $(i=1, \cdots ,N)$,
\begin{align*}
\textrm{minimize}& \quad f(X)=\sum_{i=1}^N\norm{X^\top C_iX - \mathrm{ddiag}(X^\top C_iX)}^2_F \\
\textrm{subject to}& \quad X \in \mathcal{OB}(n,p):=\{X \in \real^{n \times p} : \mathrm{ddiag}(X^\mathrm{T}X)=I_p\},
\end{align*}
where $\mathcal{S}^n$ denotes the set of all $n \times n$ symmetric matrices and $\mathrm{ddiag}(X)$ denotes a diagonal matrix whose diagonal elements are those of $X$.
\end{problem}
In the experiments, we set $N=10$, $n=100$ and $p=5$ and generated ten matrices $B_i \in \real^{n \times n}$ $(i=1,2, \cdots ,10)$ with randomly chosen elements by using \texttt{numpy.random.randn}. Then, we set symmetric matrices $C_i \in \mathcal{S}^n$ as $C_i := (B_i + B_i^\top) / 2$ $(i=1,2, \cdots ,10)$.

The experiments used a MacBook Air (2017) with a 1.8 GHz Intel Core i5, 8 GB 1600 MHz DDR3 memory, and version 10.14.5 of the macOS Mojave operating system. The algorithms were written in Python 3.7.6 with the NumPy 1.19.0 package and the Matplotlib 3.2.2 package. We solved the above four problems 100 times with each algorithm, that is, 400 times in total. If the stopping condition,
\begin{align*}
\norm{{\grad}f(x_k)}_{x_k}<10^{-6}
\end{align*}
was satisfied, we determined that a sequence had converged to an optimal solution. We compared seven Riemannian conjugate gradient methods, i.e., FR, DY, PRP, HS, HZ, Hybrid1, and Hybrid2 methods, and two line search algorithms, i.e., Algorithms \ref{alg:BacktrackArmijo} and \ref{alg:LineSearch}. In the HZ method, we set $\mu = 2$. In the Armijo condition \eqref{eq:Armijo} and the second condition of the strong Wolfe conditions \eqref{eq:sWolfe}, we set $c_1 = 10^{-4}$ and $c_2 = 0.9$. In Algorithm \ref{alg:BacktrackArmijo}, we set the scalars as $\alpha_\mathrm{hi} = 1$ and $\rho = 0.5$. In Algorithm \ref{alg:LineSearch}, we set the scalar as $\alpha_0 = 1$, and in step 11, we set $\alpha_i = 2\alpha_{i-1}$ (see \texttt{scipy.optimize.line\_search}).

For comparison, we calculated the performance profile \cite{dolan2002benchmarking}. The performance profile $P_s : \real \rightarrow [0, 1]$ is defined as follows: let $\mathcal{P}$ and $\mathcal{S}$ be the set of problems and solvers, respectively. For each $p \in \mathcal{P}$ and $s \in \mathcal{S}$, we defined
\begin{align*}
t_{p,s} := (\text{iterations or time required to solve problem }p\text{ by solver }s).
\end{align*}
Furthermore, we defined the performance ratio $r_{p,s}$ as
\begin{align*}
r_{p,s} := \frac{t_{p,s}}{\min_{s^\prime \in \mathcal{S}}t_{p,s^\prime}}
\end{align*}
and defined the performance profile, for all $\tau \in \real$, as
\begin{align*}
P_s(\tau) := \frac{\#\{p \in \mathcal{P} : r_{p,s} \leq \tau \}}{\#\mathcal{P}},
\end{align*}
where $\# S$ denotes the number of elements of a set $S$.

Figure \ref{fig1} plots the performance profiles of each algorithm by using Algorithm \ref{alg:BacktrackArmijo} to determine the step size. In particular, Figure \ref{fig1} (a) and (b) plot the performance profiles versus the number of iterations and the elapsed time, respectively. They show that the HZ method solved the most problems, which is about the same number as the Hybrid1 method solved. In particular, Hybrid1 solved more problems than the other methods in fewer iterations and less time. It can also be seen that Hybrid2 is not compatible with Algorithm \ref{alg:BacktrackArmijo}.

\begin{figure}[htbp]
\centering
\subfigure[iteration]{\includegraphics[width=100mm]{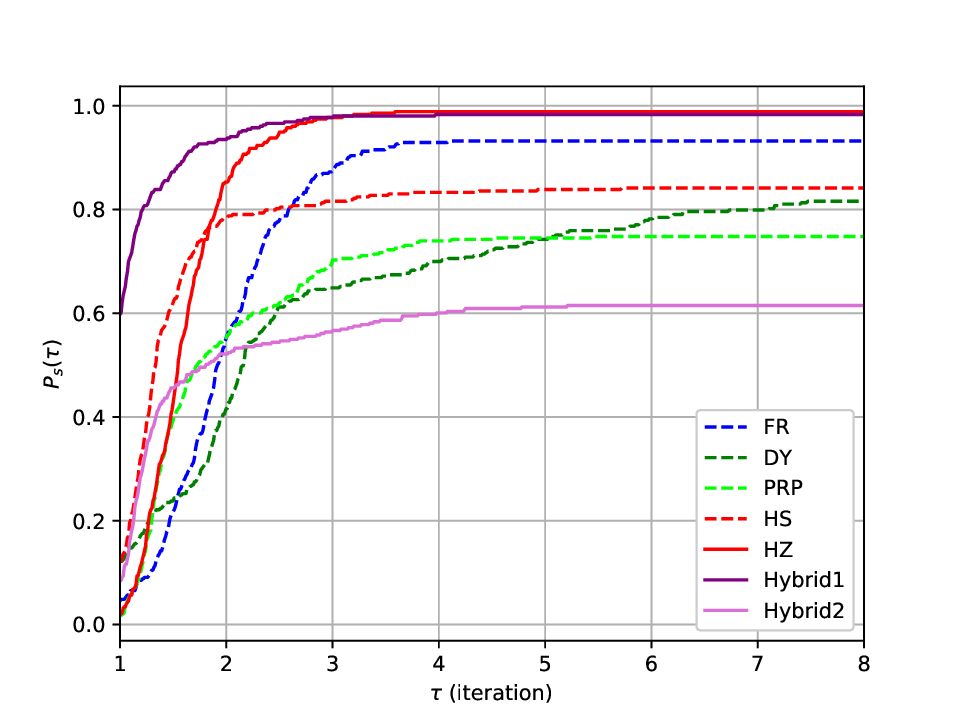}}
\subfigure[elapsed time]{\includegraphics[width=100mm]{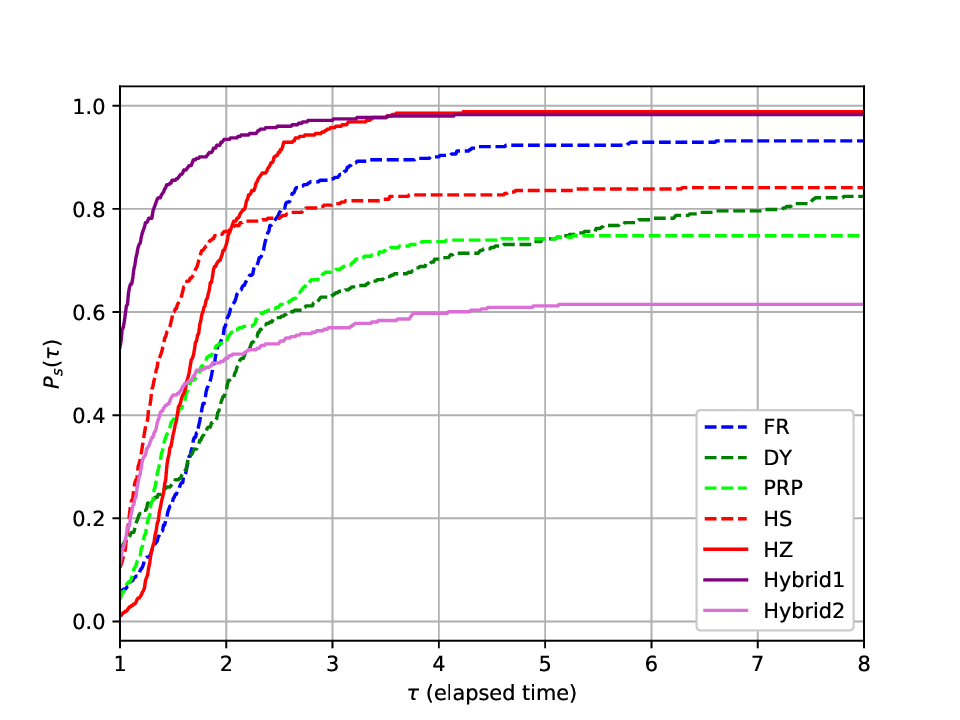}}
\caption{Performance profiles of each algorithm versus the number of iterations (a) and the elapsed time (b) by using Algorithm \ref{alg:BacktrackArmijo} to determine the step size. \label{fig1}}
\end{figure}

Figure \ref{fig2} (resp. Figure \ref{fig3}) plots the performance profiles of each algorithm by using Algorithm \ref{alg:LineSearch} to find the step size satisfying the Wolfe conditions (resp. the strong Wolfe conditions). In particular, (a) and (b) of these figures plot the performance profiles versus the number of iterations and the elapsed time, respectively. In Figure \ref{fig2}, Hybrid2 solved the second-largest number of problems, and in Figure \ref{fig3}, it solved the third-largest number. Unlike the case of using Algorithm \ref{alg:BacktrackArmijo}, Hybrid2 performed well when using Algorithm \ref{alg:LineSearch}. It can be seen that the PRP and HS methods have about the same performance, and the FR and DY methods have about the same performance.

\begin{figure}[htbp]
\centering
\subfigure[iteration]{\includegraphics[width=100mm]{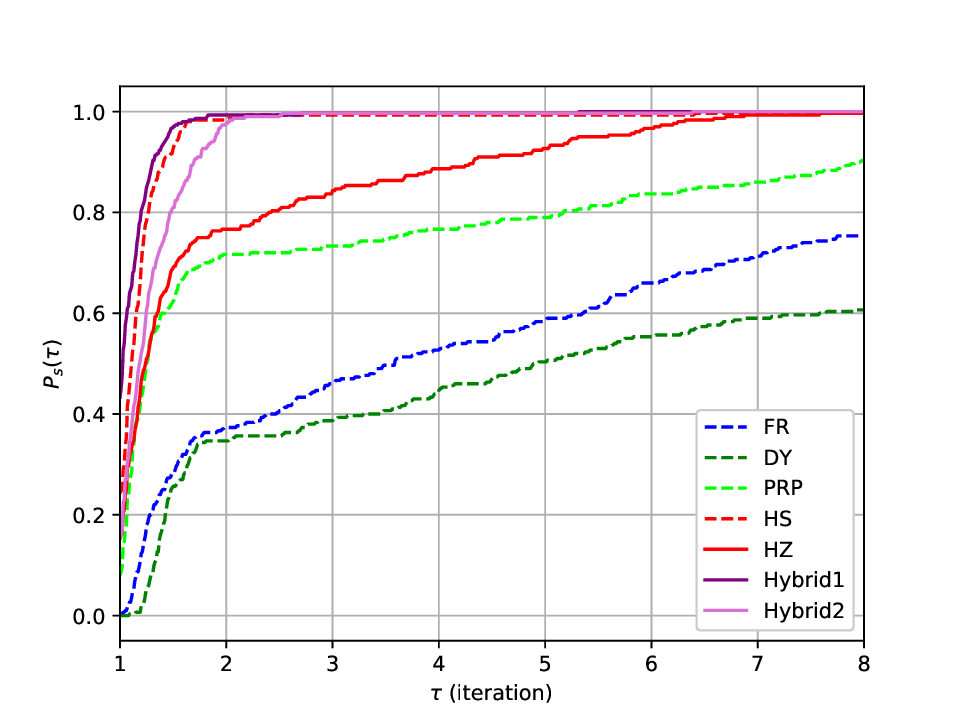}}
\subfigure[elapsed time]{\includegraphics[width=100mm]{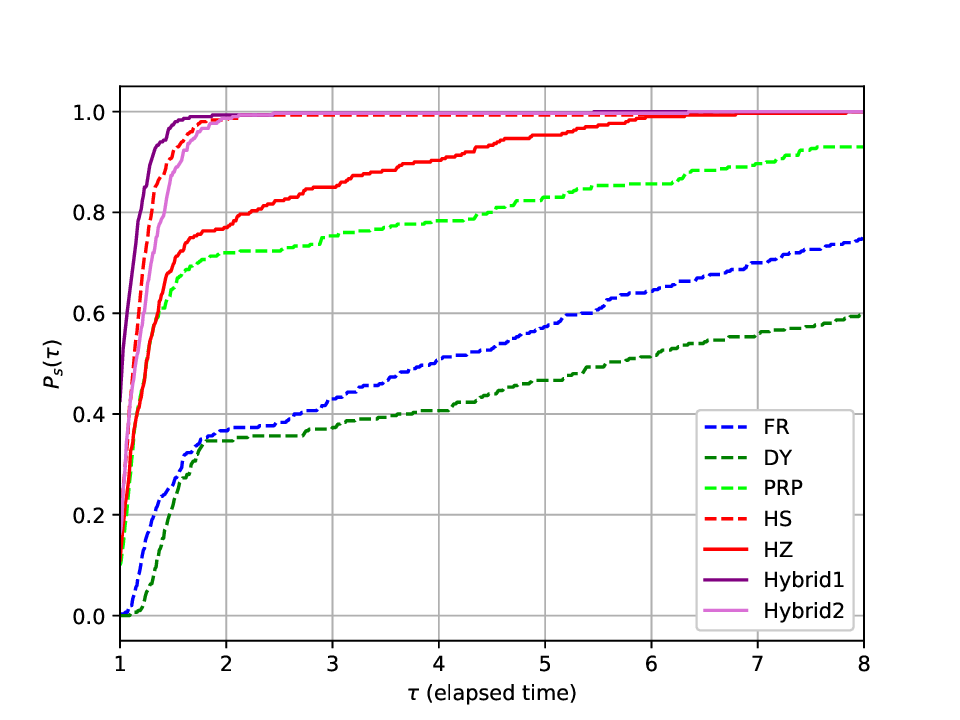}}
\caption{Performance profiles of each algorithm versus the number of iterations (a) and the elapsed time (b) by using Algorithm \ref{alg:LineSearch} to determine the step size satisfying the Wolfe conditions. \label{fig2}}
\end{figure}

\begin{figure}[htbp]
\centering
\subfigure[iteration]{\includegraphics[width=100mm]{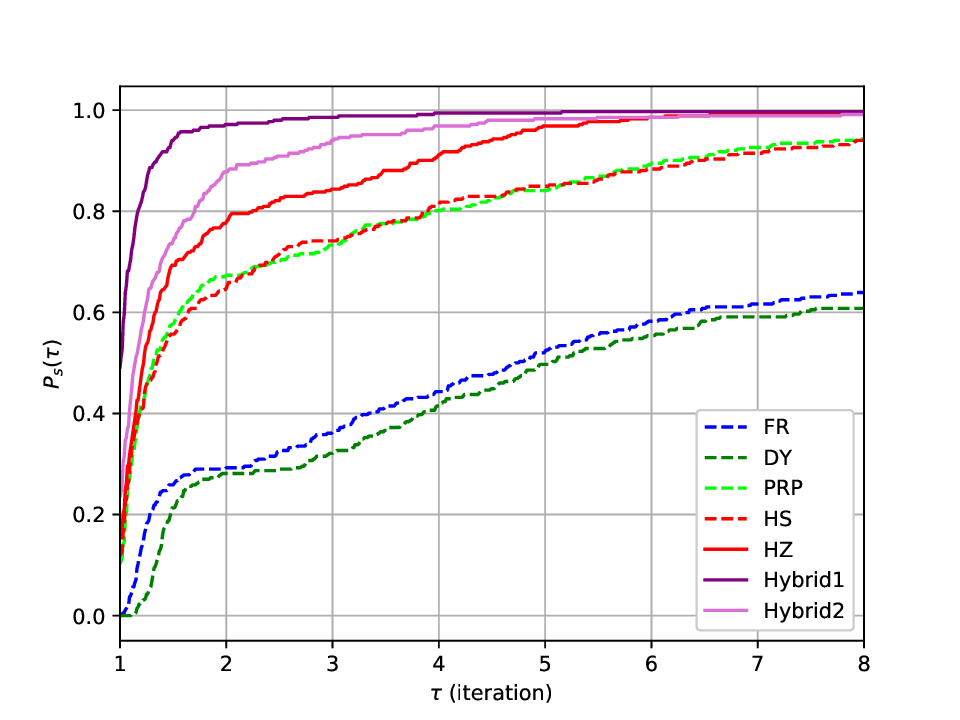}}
\subfigure[elapsed time]{\includegraphics[width=100mm]{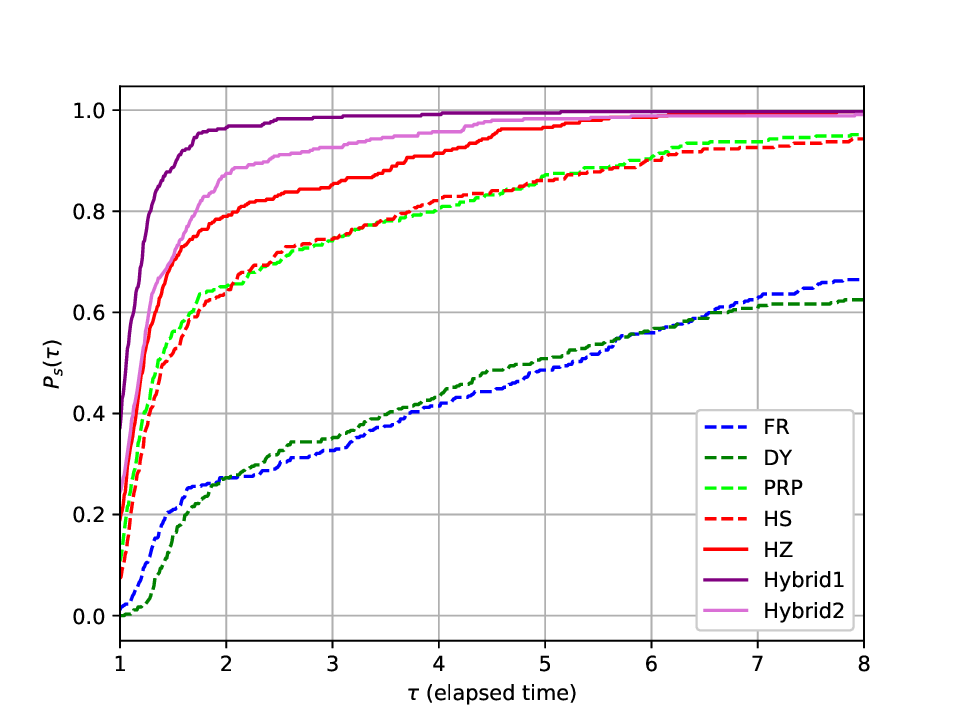}}
\caption{Performance profiles of each algorithm versus the number of iterations (a) and the elapsed time (b) by using Algorithm \ref{alg:LineSearch} to determine the step size satisfying the strong Wolfe conditions. \label{fig3}}
\end{figure}

Figure \ref{fig4} plots the performance profiles of the HZ, Hybrid1 and Hybrid2 methods by using Algorithms \ref{alg:BacktrackArmijo} and \ref{alg:LineSearch} to determine the step size satisfying the Armijo, Wolfe, and strong Wolfe conditions. In particular, Figure \ref{fig4} (a) and (b) plots the performance profile versus the number of iterations and elapsed time, respectively. Figure \ref{fig4} (a) shows that when Algorithm \ref{alg:LineSearch} is used, all methods solve the problem in fewer iterations than in the case of using Algorithm \ref{alg:BacktrackArmijo}. It can be seen from Figure \ref{fig4} (b) that Algorithm \ref{alg:LineSearch} often takes a long time to execute.

\begin{figure}[htbp]
\centering
\subfigure[iteration]{\includegraphics[width=100mm]{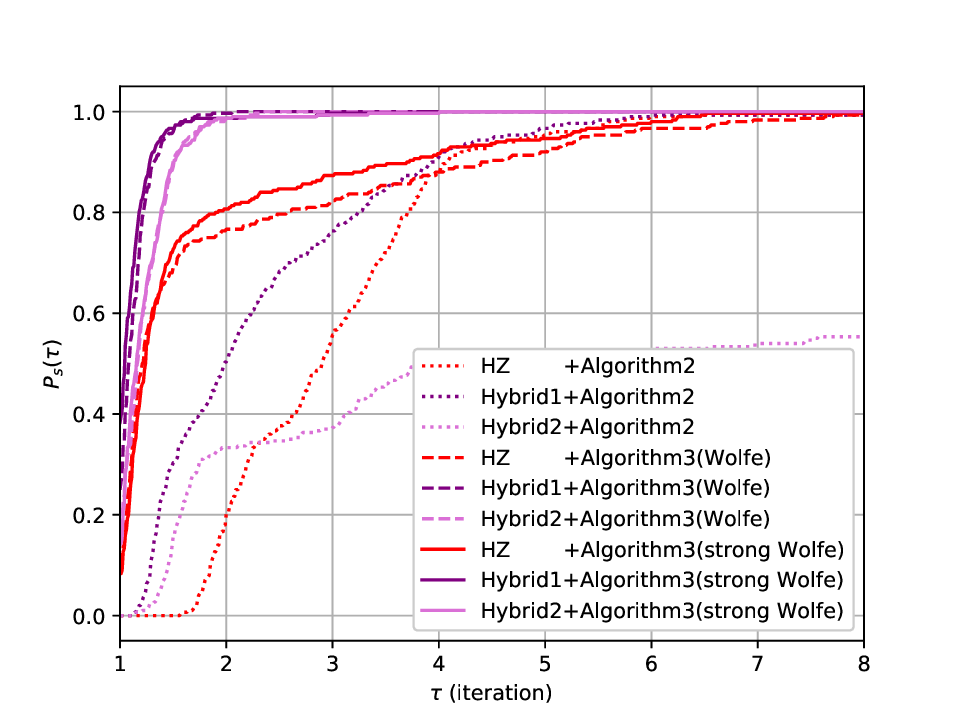}}
\subfigure[elapsed time]{\includegraphics[width=100mm]{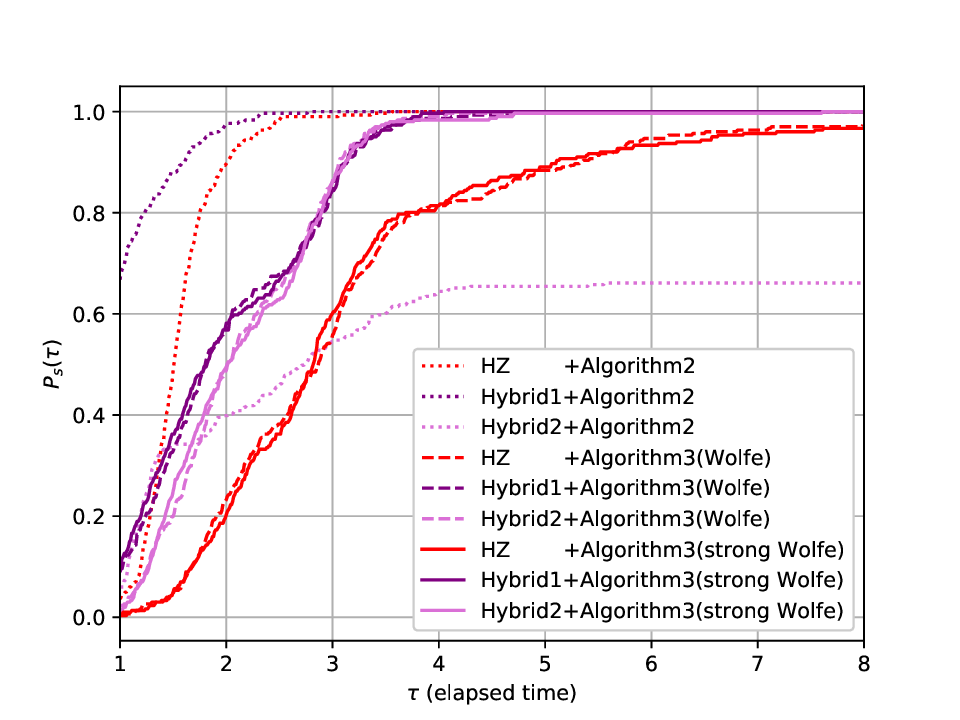}}
\caption{Performance profiles of the HZ, Hybrid1 and Hybrid2 methods versus the number of iterations (a) and the elapsed time (b) by using Algorithms \ref{alg:BacktrackArmijo} and \ref{alg:LineSearch} to determine the step size satisfying the Armijo, Wolfe, and strong Wolfe conditions. \label{fig4}}
\end{figure}

It can be seen from Figures \ref{fig1}--\ref{fig4} that Hybrid1 performed the best in all cases, but the HZ method also performed well. It can also be seen that the performances of the Riemannian conjugate gradient methods depend greatly on the type of line search used. In particular, the proposed methods (i.e., HZ and Hybrid2) performed well with Algorithm 3.

\section{Conclusion}~\label{sec:conclusion}
We generalized two nonlinear conjugate gradient methods, i.e., the HZ and Hybrid2 methods. We proved that the Hybrid2 method \eqref{eq:Hybrid2} satisfies the sufficient descent condition and converges globally under the strong Wolfe conditions. In addition, we proved that the HZ method \eqref{eq:HZ} satisfies the sufficient descent condition regardless of the type of line search used. In addition, we reviewed two kinds of line search algorithm, i.e., Algorithm \ref{alg:BacktrackArmijo} and \ref{alg:LineSearch}. In numerical experiments, we showed that the HZ and Hybrid1 methods perform well. Moreover, we showed that the performance of Riemannian conjugate gradient methods depends on the type of line search used. Hybrid2 performs better with a step size computed by Algorithm 3, as the convergence analysis guarantees. Meanwhile, the numerical results showed that the HZ method converges quickly without depending on the line search conditions. Hence, the HZ method is good for solving Riemannian optimization problems from the viewpoints of both theory and practice.

\section{acknowledgements}
We are sincerely grateful to the Editor-in-Chief, the anonymous associate editor, and the two anonymous reviewers for helping us improve the original manuscript. This work was supported by a JSPS KAKENHI Grant, Number JP18K11184.

\bibliographystyle{spmpsci_unsrt}
\bibliography{biblib}

\end{document}